\documentclass[11pt]{article}
\usepackage{a4wide}

\usepackage[a4paper, margin=2.3cm]{geometry}
\usepackage{amsmath,amssymb,amsthm}
\usepackage{color}
\usepackage{parskip}
\usepackage{hyperref}
\usepackage{parskip}
\usepackage{setspace}
\usepackage{graphicx}
\usepackage{enumitem}

\newtheorem{theorem}{Theorem}[section]
\newtheorem{remark}{Remark}[section]

\newtheorem{lemma}[theorem]{Lemma}

\newtheorem{definition}[theorem]{Definition}

\newtheorem*{definition*}{Definition}

\allowdisplaybreaks

\begin{document}
\title{A discretized point-hyperplane incidence bound in $\mathbb{R}^d$}

\author{Thang Pham \thanks{University of Science, Vietnam National University, Hanoi. Email: \href{mailto:phamanhthang.vnu@gmail.com}{phamanhthang.vnu@gmail.com}}\and Chun-Yen Shen\thanks{Department of Mathematics, National Taiwan University. Email: \href{mailto:cyshen@math.ntu.edu.tw}{cyshen@math.ntu.edu.tw}}\and Nguyen Pham Minh Tri \thanks{Ho Chi Minh University of Education. Email: \href{mailto: nguyenphamminhtri104@gmail.com}{nguyenphamminhtri104@gmail.com}}}
\date{}
\maketitle
\begin{abstract}
Let $P$ be a $\delta$-separated $(\delta, s, C_P)$-set of points in $B(0, 1)\subset \mathbb{R}^d$ and $\Pi$ be a $\delta$-separated $(\delta, t, C_\Pi)$-set of hyperplanes intersecting $B(0, 1)$ in $\mathbb{R}^d$. Define 
    \[I_{C\delta}(P, \Pi)=\#\{(p, \pi)\in P\times \Pi\colon p\in \pi(C\delta)\}.\] Suppose that $s, t\ge \frac{d+1}{2}$,
    then we have $I_{C\delta}(P, \Pi)\lesssim \delta |P||\Pi|$. The main ingredient in our argument is a measure theoretic result due to Eswarathansan, Iosevich, and Taylor (2011) which was proved by using Sobolev bounds for generalized Radon transforms. Our result is essentially sharp, a construction will be provided and discussed in the last section.
\end{abstract}
\section{Introduction}
We start with the following two definitions.
\begin{definition}[$(\delta, s, C)$ set]\label{df:our}
Let $0\le s <\infty$ and $\delta\in (0,1)$ and a constant $C>0$. Given a metric space $(X, d)$, a bounded set $E\subset X$ is called $(\delta, s, C)$-set if for every $\delta\le r\le 1$ and for all ball $B\subset X$ of radius $r$, we have 
\[|E\cap B|_\delta\le C r^s|E|_{\delta},\]
where $|E|_\delta$ denotes the $\delta$-covering of $E$ in the space $(X, d)$. 
\end{definition}
We note that this definition is slightly different compared to the classical one introduced by Katz and Tao in \cite{KatzTao}. 
\begin{definition}[Katz-Tao $(\delta, s, C)$-set]\label{def:kt}
    Let $0\le s <\infty$ and $\delta\in (0,1)$ and a constant $C>0$. Given a metric space $(X, d)$, a bounded set $E\subset X$ is called Katz-Tao $(\delta, s, C)$-set if for every $\delta\le r\le 1$ and for all ball $B\subset X$ of radius $r$, we have 
\[|E\cap B|_\delta\le C \left(\frac{r}{\delta}\right)^s.\]
 
\end{definition}
Note that if $|E|_{\delta}\sim \delta^{-s}$, then the two above definitions are equivalent. 

 Let $P$ be a $\delta$-separated $(\delta, s, C_P)$-set of points in $\mathbb{R}^d$ and $\Pi$ be a $\delta$-separated $(\delta, t, C_\Pi)$-set of hyperplanes intersecting $B(0, 1)$ in $\mathbb{R}^d$. The number of discretized point-plane incidences between $P$ and $\Pi$ is defined by
    \[I_{C\delta}(P, \Pi)=\#\{(p, \pi)\in P\times \Pi\colon p\in \pi(C\delta)\},\] 
where $\pi(C\delta)$ denotes the $C\delta$ neighborhood of the hyperplane $\pi$. In this paper, we treat $C_P$, $C_\Pi$, and other constants as absolutely positive bounded constants. 

Our initial motivation comes from the following recent theorem due to Orponen, Shmerkin, and Wang \cite{motmot}, and Fu and Ren \cite{FuRen} in two dimensions. 
\begin{theorem}\label{twodimension}
    Let $0\le s, t\le 2$. Then, for every $\epsilon>0$, there exists $\delta_0=\delta_0(\epsilon)>0$ such that the following holds for all $\delta\in (0, \delta_0]$. Let $P\subset B(0, 1)\subset \mathbb{R}^2$ be a $\delta$-separated $(\delta, s, \delta^{-\epsilon})$-set of points and $\mathcal{T}$ be a $\delta$-separated $(\delta, t, \delta^{-\epsilon})$-set of tubes intersecting $B(0, 1)$. 
    \begin{enumerate}
    \item If $1\ge t\ge s$ or $1\ge s\ge t$, then 
    \[I_{\delta}(P, \mathcal{T})\lesssim |P||\mathcal{T}|\delta^{\frac{st}{s+t}-O(\epsilon)}.\]
     \item If $t\ge 1\ge s\ge t-1$, then 
    \[I_{\delta}(P, \mathcal{T})\lesssim |P||\mathcal{T}|\delta^{\frac{st}{1+s}-O(\epsilon)}\]
\item If $s\ge 1\ge t\ge s-1$, then 
\[I_{\delta}(P, \mathcal{T})\lesssim |P||\mathcal{T}|\delta^{\frac{st}{1+t}-O(\epsilon)}.\]    
        \item If $t>1$ and $s>1$, then 
        \[I_{\delta}(P, \mathcal{T})\lesssim |P||\mathcal{T}|\delta^{\kappa(s+t-1)-O(\epsilon)},~\kappa=\min \{1/2, ~1/(s+t-1)\}.\]
    \end{enumerate}
\end{theorem}
Throughout this paper, by $X\lesssim Y$ we mean that $X\le CY$ for some constant $C$, and $X\sim Y$ if $X\lesssim Y\lesssim X$. 

The main purpose of this paper is to extend Theorem \ref{twodimension} to higher dimensions, namely, in $\mathbb{R}^d$ with $d\ge 3$. Our first result is the following. 
\begin{theorem}\label{main-theorem}Let $C>0$, $0\le s, t\le d$. There exists $\delta_0=\delta_0(C, s, t)>0$ such that the following holds for $\delta\in (0, \delta_0)$. 
    Let $P$ be a $\delta$-separated $(\delta, s, C_P)$-set of points in $B(0, 1)\subset \mathbb{R}^d$ and $\Pi$ be a $\delta$-separated $(\delta, t, C_\Pi)$-set of hyperplanes intersecting $B(0, 1)$ in $\mathbb{R}^d$. Define 
    \[I_{C\delta}(P, \Pi)=\#\{(p, \pi)\in P\times \Pi\colon p\in \pi(C\delta)\}.\] Suppose that $s, t > \frac{d+1}{2}$,
    then we have the sharp estimate that $I_{C\delta}(P, \Pi)\lesssim \delta |P||\Pi|$. 
\end{theorem}
In the above theorem, the conditions $s > \frac{d+1}{2}$ and $t > \frac{d+1}{2}$ are required in the proof. When $s, t$ are small, using an elementary geometric argument, we are able to prove the following non-trivial result. 
\begin{theorem}\label{CSbound}Let $C>0$, $0\le s, t\le d$. There exists $\delta_0=\delta_0(C, s, t)>0$ such that the following holds for $\delta\in (0, \delta_0)$. 
    Let $P$ be a $(\delta,s,C_P)$ set of points in $B(0, 1)\subset \mathbb{R}^d$ and $\Pi$ be a $(\delta,t,C_\Pi)$ set of hyperplanes intersecting $B(0, 1)$ in $\mathbb{R}^d$ with $d \geq 3$. We further assume that $s-d+2 >0$. Then, for any $\epsilon >0$, we have
    \[I_{C\delta}(P,\Pi)\lesssim  |P|\cdot|\Pi|\cdot \delta^{f(t)(s-d+2) - \epsilon},\] where  \[
        f(t) =
        \begin{cases}
        1/2,  ~&\text{ if } t\ge 1\\
        \dfrac{t}{1+t}, ~&\text{ if } t<1
        \end{cases}.\]
\end{theorem}
\begin{remark}
    If $t-d+2 > 0$, then we can use the dual arguments to obtain a similar result. While Theorem \ref{main-theorem} is optimal, we do not have any constructions on the sharpness of this theorem. 
\end{remark}

\paragraph{Main ideas and Comparisons:} We first discuss the main idea in the proof of Theorem \ref{twodimension}. Observe that if $P$ is a $(\delta, s, C)$-set, then $P$ can be covered by at most $|P|\delta^{s-O(\epsilon)}$ Katz-Tao $(\delta, s)$-sets. A detailed proof can be found in \cite[Lemma 3.5]{motmot}. With this observation, one can apply Theorem 1.4 and Theorem 1.5 due to Fu and Ren in \cite{FuRen}. To prove these two theorems, Fu and Ren used a geometric argument and some earlier results due to Guth, Solomon, and Wang \cite{GSW} and a generalization due to Bradshaw \cite{PB}. 

In higher dimensions $d\ge 3$, to prove Theorem \ref{main-theorem}, we use a completely different approach. More precisely, we use a measure theoretic result due to Eswarathansan, Iosevich, and Taylor \cite{Alex}, which was proved by using Sobolev bounds for generalized Radon transforms. The proof of Theorem \ref{CSbound} is a combination of Cauchy-Schwartz  argument and geometric results due to Hera, Keleti, and Mathe \cite{HKM}. 

We also want to add a remark that the approach in the proof of Theorem \ref{main-theorem} is similar to the mechanism introduced by Iosevich, Jorati, and Laba \cite{III} when they studied incidences between a set of ``tubes" and a homogeneous set of points with the same size. Since our proof uses Sobolev bounds for generalized Radon transforms, it can be generalized to a more general form, i.e. the equation of hyperplanes $x_d=a_1x_1+\cdots+a_{d-1}x_{d-1}+a_d$ can be replaced by $\Psi(a_1, \ldots, a_d, x_1, \ldots, x_d)=0$, where the function $\Psi$ satisfies the Phong-Stein curvature condition (\ref{P-S-C}) below.

If the set $\Pi$ of planes  only needs to satisfy the property that it is $\delta$-separated, a recent work of Dabrowski, Orponen, and Villa \cite{DOV} for the case of $(d-1)$-hyperplanes tells us that 
\begin{equation}\label{eq:DOV}I_{C\delta}(P, \Pi)\lesssim \delta^{-\epsilon}\delta^{\frac{(d-1)(s+1-d)}{2d-1-s}}|P||\Pi|^{\frac{d-1}{2d-1-s}},\end{equation}
where $P$ is a $\delta$-separated $(\delta, s, C_P)$-set with $s>1$.

This result is weaker than Theorem \ref{main-theorem} when $|\Pi|\le \delta^{-d}$. 

We now compare the estimate (\ref{eq:DOV}) and Theorem \ref{CSbound}. Theorem \ref{CSbound} states that if $s-d+2>0$, then 
\[I_{C\delta}(P,\Pi)\lesssim  |P|\cdot|\Pi|\cdot \delta^{f(t)(s-d+2) - \epsilon},\] where  \[
        f(t) =
        \begin{cases}
        1/2,  ~&\text{ if } t\ge 1\\
        \dfrac{t}{1+t}, ~&\text{ if } t<1
        \end{cases}.\]
Assume $t\ge 1$, then by a direct computation, this bound is stronger than that of (\ref{eq:DOV}) when 
\[\delta^{-t}\lesssim |\Pi|\le \delta^{\frac{M(2d-1-s)}{d-s}},\]
where 
\[M=\frac{(d-1)(s+1-d)}{2d-1-s}-\frac{s-d+2}{2}.\]
This range is non-empty when $2t\le s+1<d$.

Assume $t<1$, then Theorem \ref{CSbound} is stronger than that of (\ref{eq:DOV}) when 
\[\delta^{-t}\lesssim |\Pi|\le \delta^{\frac{M'(2d-1-s)}{d-s}},\]
where 
\[M'=M-\left(\frac{t}{t+1}-\frac{1}{2}\right)(s-d+2).\]
This range is non-empty when $ d-2<s<d-1$.

\section{Proof of Theorem \ref{main-theorem}}
We first recall some notations that can be found in \cite{DOV}. Let $\mathcal{A}(d, n)$ be the set of $n$-dimensional affine subspaces in $\mathbb{R}^d$. The metric $d_{\mathcal{A}}$ defined on $\mathcal{A}(d, n)$ (\cite[page 53]{M95}) is defined by 
\[d_{\mathcal{A}}(V, W)=||\pi_{V_0}-\pi_{W_0}||_{op}+|a-b|,\]
where $V=V_0+a$ and $W=W_0+b$, $V_0, W_0\in \mathcal{G}(d, n)$ (the set of $n$-dimensional subspaces), $a\in V_0^\perp$, $b\in W_0^\perp$, and $||\cdot ||_{op}$ is the operator norm.

Let $\gamma_1$ and $\gamma_2$ be two $(d-1)$-planes defined by 
\[\gamma_1\colon x_d=a_1x_1+\cdots+a_{d-1}x_{d_1}+a_d,\]
and 
\[\gamma_2\colon x_d=b_1x_1+\cdots+b_{d-1}x_{d_1}+b_d.\]
The following formula is more useful in practise
\[d_\mathcal{A}\left(\gamma_1,\gamma_2\right) = \left|\dfrac{(a_1\ldots,a_{d-1},-1)}{\left|(a_1\ldots,a_{d-1},-1\right|}-\dfrac{(b_1\ldots,b_{d-1},-1)}{\left|(b_1\ldots,b_{d-1},-1\right|}\right|+\left|\dfrac{a_d}{\left|(a_1\ldots,a_{d-1},-1\right|}-\dfrac{b_d}{\left|(b_1\ldots,b_{d-1},-1)\right|}\right|.\]
This can be proved by a direct computation or can be found in \cite[Lemma 2.5]{HPab}. 

For each plane defined by $x_d=a_1x_1+\cdots+a_{d-1}x_{d-1}+a_d$, we call the point $(a_1, \ldots, a_d)$ its dual point in $\mathbb{R}^d$. 

Let $D\colon \mathbb{R}^d\to \mathcal{A}(d, d-1)$ be the map defined by 

\[D \colon (x_1,\ldots,x_d) \mapsto \left\lbrace y_d=\sum_{i=1}^{d-1} x_iy_i+x_d\right\rbrace.\]
There are some properties of $D$ we should keep in mind, namely, the restriction of $D$ to the ball $B(R)$, $0<R<\infty$, is bilipschitz onto its image, and the bilipschitz constant depends only on $R$ and $d$. Thus, $D$ is injective. The inverse map of $D$, denoted by $D^*\colon \mathtt{im}(D)\to \mathbb{R}^d$, defined by 
\[D(x_1, \ldots, x_d)\mapsto (-x_1, \ldots, -x_{d-1}, x_d).\]
There exists a dimensional constant $0<r_d<1$ such that the restriction of $D^*$ to $B(V_0, r_d)$ is bilipschitz onto its image, here $V_0=D(0, \ldots, 0)$. 

With these notations, the next two lemmas can be proved with standard arguments. 
\begin{lemma}\label{lm1}
    Let $S$ be a set of $\delta$-separated hyperplanes intersecting the unit ball $B(0, 1)$. Then the set of corresponding dual points is also $c\delta$-separated for some absolute constant $c$.
\end{lemma}
\begin{proof}
If two hyperplanes $\gamma_1$ and $\gamma_2$ are $\delta$-separated with the metric
\begin{align*}
d_\mathcal{A}\left(\gamma_1,\gamma_2\right) = \left|\dfrac{(a_1\ldots,a_{d-1},-1)}{\left|(a_1\ldots,a_{d-1},-1\right|}-\dfrac{(b_1\ldots,b_{d-1},-1)}{\left|(b_1\ldots,b_{d-1},-1\right|}\right|+\left|\dfrac{a_d}{\left|(a_1\ldots,a_{d-1},-1\right|}-\dfrac{b_d}{\left|(b_1\ldots,b_{d-1},-1)\right|}\right|,
\end{align*}
where $\gamma_1$ and $\gamma_2$ are respectively defined by equations
  \[(\gamma_1): y_d=\sum_{i=1}^{d-1}a_iy_i+a_d,\]
    and 
\[(\gamma_2): y_d=\sum_{i=1}^{d-1}b_iy_i+b_d,\]
then we want to prove that the usual Euclidean distance between two points $(a_1, \ldots, a_d)$ and $(b_1, \ldots, b_d)$ is at least $c\delta$ for some absolute constant $c$. 
    For the sake of simplicity, we will denote $u=(a_1\ldots,a_{d-1},-1)$ and $v =(b_1\ldots,b_{d-1},-1) $. Then clearly $|u|,|v| \ge 1$. Since two hyperplanes $\gamma_1, \gamma_2$ are $\delta$-separated, we have that $d_\mathcal{A}\left(\gamma_1,\gamma_2\right) \ge \delta$. With these assumptions, we attempt to prove that 
    \[ \sqrt{\sum_{i=1}^d (a_i-b_i)^2} \gtrsim \delta\]
    It should be noted that, if we view $a_d,b_d$ as vectors $(0,\ldots,a_d), (0,\ldots,b_d)$ respectively, then the usual Euclidean distance between two points $(a_1, \ldots, a_d)$ and $(b_1, \ldots, b_d)$ can be rewritten as $\sqrt{|u-v|^2+|a_d-b_d|^2}$.

Thus, the goal is to prove the following inequality
\begin{align}\label{inq}
    \left|\dfrac{u}{|u|}-\dfrac{v}{|v|}\right|+\left|\dfrac{a_d}{|u|}-\dfrac{b_d}{|v|}\right|\lesssim  \sqrt{|u-v|^2+|a_d-b_d|^2}.
\end{align}

Instead of proving \ref{inq}, we prove a stronger version
\begin{align}
    \left|\dfrac{u}{|u|}-\dfrac{v}{|v|}\right|+\left|\dfrac{a_d}{|u|}-\dfrac{b_d}{|v|}\right|\lesssim |u-v| +|a_d-b_d|.
\end{align}\label{inq1}
We first prove 
\[\left|\dfrac{u}{|u|}-\dfrac{v}{|v|}\right| \le |u-v|. \]
Squaring both sides gives
\[|u|^2+|v|^2-2u\cdot v\ge 2-\dfrac{2u\cdot v}{|u||v|}.\] 
This is reduced to
\[2|u||v|-2u\cdot v \ge 2-\dfrac{2u\cdot v}{|u||v|}.\]
We denote $|u||v|=x, u\cdot v=y$, then the above can be represented as
\begin{align*}
           (x-1)(x-y) \ge 0.
\end{align*}
This inequality is true since $|u|,|v| \ge 1$ and $x \ge y$ by the Cauchy-Schwarz inequality.

Now we estimate $\left|\frac{a_d}{|u|}-\frac{b_d}{|v|}\right|$. Since all hyperplanes intersect $B(0, 1)$, one has $\frac{|a_d|}{|u|}$ and $\frac{|b_d|}{|v|}$, i.e the distances from the origin to the hyperplanes $\gamma_1,\gamma_2$, are not larger than 1. Then
\[\left|\dfrac{a_d}{|u|}-\dfrac{b_d}{|v|}\right| = \left|\dfrac{a_d|v|-a_d|u|+a_d|u|-b_d|u|}{|u||v|}\right| \le |u-v| +|a_d-b_d|.\]
This completes the proof.
\end{proof}
\begin{lemma}\label{lm2}
    Let $\Pi$ be a set of $\delta$-separated $(\delta, t, C_{\pi})$-hyperplanes intersecting $B(0, 1)$. Then the set of dual points is $c\delta$-separated $(\delta, t, c')$-set for some absolute constants $c, c'>0$.
\end{lemma}
\begin{proof}
     We denote $\Pi^*$ the set of dual points of hyperplanes $\Pi$. It is immediate that the set $\Pi^*$ is $c\delta$-separated by Lemma \ref{lm1}. Hence, we only need to prove that, for arbitrary ball $B(x,r)$, we have
    \[|\Pi^*\cap B(x, r)|_\delta \le c' r^t|\Pi^*|_\delta, \delta\le r\le 1,\]
    for some constant $c'$ that will be chosen later.  It is sufficient to prove that 
    \[|\left\lbrace  x' \in \Pi^*: d(x',x) \le r\right\rbrace| \le c'r^t|\Pi^*|_{\delta}, \delta\le r\le 1.\]
As mentioned earlier, the map $D$ is bilipschitz onto its image, we have $|\left\lbrace  x' \in \Pi^*: d(x',x) \le r\right\rbrace|$ is at most
    \[\left|\left\lbrace  D(x')\in \Pi : d_\mathcal{A}(D(x'),D(x)) \le K_Dr\right\rbrace\right| \le C_\pi K_D^t r^t|\Pi^*|_\delta, \delta\le r\le 1.\]
    Choose $c'=C_\pi K_D^{t}$, then the lemma follows. Note that $K_D^t$ can be replaced by some constant that does not depend on $t$ since $t\in (0, d)$. 
    \end{proof}
\subsection{Sobolev bounds for generalized Radon transforms and consequences}

In this section, we recall some known results that make use the boundedness of general Radon transforms. Let $g\colon \mathbb{R}^d\to \mathbb{R}$ be a Schwartz function, $t\in \mathbb{R}$, $\psi\colon \mathbb R^d \times \mathbb R^d \rightarrow \mathbb R$ be a smooth cut-off function, and   $ \Psi(\mathbf{x}, \mathbf{y}) : \mathbb R^d \times \mathbb R^d \rightarrow \mathbb R$ be a  smooth function with some suitable assumptions. We define
\[T_{\Psi_t}g(\mathbf{x}):=\int_{\{\Psi(\mathbf{x}, \mathbf{y})=t\}}g(\mathbf{y})\psi(\mathbf{x}, \mathbf{y})d\sigma_{\mathbf{x}, t}(\mathbf{y}),\]
where $d\sigma_{\mathbf{x}, t}$ is the Lebesgue measure on the set $\{\mathbf{y}\colon \Psi(\mathbf{x}, \mathbf{y})=t\}.$

We denote the usual $L^2$-Sobolev space of $L^2$ functions with $s$ generalized derivatives in $L^2(\mathbb{R}^d)$ by $L^2_s(\mathbb{R}^d)$. The following theorem was proved by Eswarathasan, Iosevich and Taylor in \cite{Alex}.
\begin{theorem} [\cite{Alex}, Proposition 2.2]
Let $E\subset \mathbb{R}^d$ be a compact set with $\dim_H(E)=\alpha$, and $\mu$ be the corresponding Frostman measure on $E$. Assume that $T_{\Psi_t}$ maps $L^2$ to $L^2_s$ with constants uniform in a small neighborhood of $t$, and $ d-s < \alpha < d$. Then we have
$$\mu \times \mu \{(\mathbf{x}, \mathbf{y}) \in E \times E : t \leq |\Psi(\mathbf{x},\mathbf{y})| \leq t+ \epsilon\} \lesssim \epsilon.$$

\end{theorem}

\begin{remark}\label{remarks}
It can be checked from Eswarathasan-Iosevich-Taylor's proof that the same result holds for two different sets $E \times F$, namely,
\begin{equation*}\mu_E \times \mu_F \{(\mathbf{x}, \mathbf{y}) \in E \times F : t \leq |\Psi(\mathbf{x},\mathbf{y})| \leq t+\epsilon\} \lesssim \epsilon,\end{equation*}
if  $T_{\Psi_t}$ maps $L^2$ to $L^2_s$ with $ d-s < \alpha, \beta < d$, where $\alpha=\dim_H(E)$ and $\beta=\dim_H(F)$. Notice also that the sets $E$ and $F$ are not required to be A-D regular. Moreover in the proofs, the only thing that is needed associated with the set $E$ is the existence of a probably measure $\mu$ supported on $E$ that satisfies $\mu(B(x, r)) \lesssim r^{\alpha}$, the same applies for $F$.
\end{remark}
To apply the above theorem in the proof of Theorem \ref{main-theorem}, we need to recall a celebrated result of Phong and Stein \cite{PhS} stating that the operator $T_{\Psi_t}$ is uniformly bounded from $L^2$ to $L^2_s$ on a small neighborhood of $t$ with $s= \frac{d-1}{2}$ if the so-called Phong-Stein rotational curvature condition 
\begin{equation}\label{P-S-C}\det \begin{pmatrix}
0 & \nabla_\mathbf{x}\Psi \\
-\nabla_\mathbf{y}\Psi & \frac{\partial^2 \Psi}{\partial x_i \partial y_j} 
\end{pmatrix} \ne 0\end{equation} 
holds on the set $\{(\mathbf{x}, \mathbf{y})\colon \Psi(\mathbf{x},\mathbf{y})=t\}$. This and Remark \ref{remarks} imply the following theorem. 
\begin{theorem}\label{2.1.1}
Let $\mu_E, \mu_F$ be probability measures on  compact sets $E, F \subset \mathbb R^d,$ respectively, satisfying 
\[\mu_E(B(x, r))\lesssim r^\alpha, ~\mu_F(B(x, r))\lesssim r^\beta,\]
for all $x\in \mathbb{R}^d$ and $r>0$. Suppose that the Phong-Stein rotational curvature condition \eqref{P-S-C} holds for the function $\Psi$, and $d-(d-1)/2 < \alpha, \beta <d.$ Then, for $\epsilon>0,$  we have
$$\mu_E \times \mu_F \{(\mathbf{x}, \mathbf{y}) \in E \times F : |\Psi(\mathbf{x},\mathbf{y})| \leq \epsilon\} \lesssim \epsilon.$$
\end{theorem}
\subsection{Proof of Theorem \ref{main-theorem}}
Let $\Pi^*$ be the set of dual points corresponding to planes in $\Pi$. As we proved in Lemma \ref{lm2} that this set is $c\delta$-separated $(\delta, t, c')$-set. 

We first define two probability measures on $P(\delta)$ and $\Pi^*(c\delta)$, denoted by $\mu_P$ and $\mu_{\Pi}$, respectively, as follows:
\[\mu_P(X)=\frac{ | X\cap P(\delta)|}{|P(\delta)|},\]
and 
\[\mu_{\Pi^*}(X)=\frac{|X\cap \Pi^*(c\delta)|}{|\Pi^*(c\delta)|}.\]
Since $P$ is $(\delta, s, C_P)$-set and $\Pi^*$ is $(c\delta, t, c')$-set, we have
\[\mu_P(B(x, r)) \lesssim r^s,\]
and 
\[\mu_{\Pi^*}(B(x, r))\lesssim r^{t},\]
for all $x\in \mathbb{R}^d$ and $r\ge \delta$. 

For $r\le \delta$, we have $|B(x, r)\cap P(\delta)|\lesssim r^d$. On the other hand, we have $ |P(\delta)|\gtrsim \delta^{d-s},$
which follows from the assumption that $P$ is a $(\delta, s, C_P)$-set. For $r\le \delta$ and $s<d$, we have $r^{d-s}\le \delta^{d-s}$, this means that 
\[\mu_P(B(x, r))\lesssim \frac{r^d}{\delta^{d-s}}\lesssim r^s.\]
The same holds for $\Pi^*$. In other words, the two probability measures $\mu_{P}$ and $\mu_{\Pi^*}$ are Frostman measures with exponents $s$ and $t$, respectively. 

To proceed further, we may assume that all hyperplanes are defined by the equation of the form 
\begin{equation}\label{star}a_1x_1+\cdots +a_{d-1}x_{d-1}+a_d=x_d.\end{equation}
This gives that the dual points are of the form $(a_1, \ldots, a_{d-1}, 1)$.
Define 
\[\Psi(x_1, \ldots, x_d, a_1, \ldots, a_d)=a_1x_1+\cdots+a_{d-1}x_{d-1}-x_d+a_d.\]
A direct computation shows that this function satisfies the curvature condition (\ref{P-S-C}). 
We now observe that $(x_1, \ldots, x_d)\in \pi(C\delta)$, where $\pi$ is defined by (\ref{star}), if $|\Psi(x_1, \ldots, x_d, a_1, \ldots, a_d)|\le C\delta$. On the other hand, if $|\Psi(x_1, \ldots, x_d, a_1, \ldots, a_d)|\le C\delta$, then $|\Psi(U,V)|\lesssim \delta$ for all $U\in B((x_1, \ldots, x_d), \delta)$ and $V\in B((a_1, \ldots, a_d), c\delta)$, when $\delta$ is small enough. This infers that
\[\frac{1}{|P||\Pi|}|I_{\delta}(P, \Pi)|\le \mu_P\times \mu_{\Pi^*}\left\lbrace (U, V)\in P(\delta)\times \Pi^*(c\delta) \colon |\Psi(U, V)|\lesssim \delta \right\rbrace.\]

Therefore, our result is reduced to show the following.
\[\mu_P\times \mu_{\Pi^*}\left\lbrace (U, V)\in P(\delta)\times \Pi^*(c\delta)\colon |\Psi(U, V)|\lesssim \delta \right\rbrace\lesssim \delta,\]
which follows from Theorem \ref{2.1.1}.

\section{Proof of Theorem \ref{CSbound}}
In this section, we present an elementary argument to study the incidence problem. 

For $p\in P$, let $I(p)$ be the set of hyperplanes $\pi\in \Pi$ such that $p\in \pi(C\delta)$. We observe that
\[I_{C\delta}(P, \Pi)=\sum_{p\in P}|I(p)|.\]
Thus, for $x, y > 0$ and $x\ge y$, by the H\"{o}lder inequality, we have 
\[I_{C\delta}(P, \Pi)\le |P|^{\frac{x}{x+y}}\left(\sum_{p}|I(p)|^{\frac{x+y}{y}}\right)^{\frac{y}{x+y}}.\]
This implies that 
\[I_{C\delta}(P, \Pi)^{x+y}\le |P|^{x}\left(\sum_{p}|I(p)|^{\frac{x+y}{y}}\right)^{y}.\]
To proceed further, we need to estimate the sum $\sum_{p}|I(p)|^{1+x/y}$. 
We have 
\begin{align*}
    \sum_{p\in P}|I(p)|^{1+x/y}=\sum_{p\in P}\sum_{\pi\colon p\in \pi(C\delta)}|I(p)|^{x/y}=\sum_{\pi\in \Pi}J(\pi),
\end{align*}
here $J(\pi)=\sum_{p\in \pi(C\delta)}|I(p)|^{x/y}$, which can be represented as follows
\begin{align*}
    J(\pi)=\sum_{p\in \pi(C\delta)}\left(\sum_{i}|I(p)\cap J_{2^i\delta}(\pi)
|\right)^{x/y} \lesssim \sum_{p\in \pi(C\delta)} \sum_{i=1}^{\log \delta^{-1}}|I(p)\cap J_{2^i\delta}(\pi)|^{x/y},
\end{align*}
where $J_{2^i\delta}(\pi)$ is the set of hyperplanes $\pi'$ such that $d_{\mathcal{A}}(\pi, \pi')\sim 2^i\delta$.

Since the set of planes is $(\delta, t, C_\Pi)$, we have 
\[|J_{2^i\delta}(\pi)| \lesssim (2^i \delta)^t|\Pi|.\]

 \begin{lemma}\label{4.1}
 Let $\pi$ and $\pi'$ be two hyperplanes in $\mathcal{A}(d, d-1)$. Assume these two planes both intersect the unit ball $B(0, 1)$ and $d_\mathcal{A}(\pi, \pi')=w>\delta$, then the intersection $\pi(\delta)\cap \pi'(\delta)\cap B(0, 1)$ can be covered by at most $\delta^{-(d-2)}$ cubes of parameters $\frac{\delta}{w}\times \delta\times \cdots\times \delta$ in $\mathbb{R}^d$. 
 \end{lemma}
\begin{proof}
The proof is essentially a combination of a number of results from \cite{HKM}. Let $e_0=(0, \cdots, 0)$ and $\left\{ e_1,\cdots,e_d\right\}$ be the standard basis of $\mathbb{R}^d$. If we denote $H_0=\left\langle e_d \right\rangle$ then $H_0$ is a line and $\pi \cap H_0$ is a point of the form $(0,\cdots,0,a_0)$. We also denote $H_i = e_i+H_0$ so that $\pi \cap H_i$ is also a point of the form $(0,\cdots,1,\cdots,a^i)$. Let $b_i=a_i-a_0$, then we refer to $a_0$ as the vertical intercept and $b_i$ as the slopes of hyperplane $\pi$.

For each hyperplane $\pi$, we associate it to a point $x = x(\pi) = (a_0,b_1,\cdots,b_{d-1}) \in \mathbb{R}^d$ and note that the map $\varphi\colon \Pi \mapsto x(\Pi)$ is well-defined and injective. The space generated by this map is called \textbf{code space}. We endow this space with the maximum metric, namely, 
\[||x-x'||:=\max \left(|a_0-a_0'|, \max_{i=1, \ldots, d-1}\left(|b_i-b_i'|\right)\right).\]

As noted in \cite[Remark 4.2]{HKM}, the maximum metric on the code space is strongly equivalent to the metric between hyperplanes $d_\mathcal{A}$, in the sense that $||\varphi(\pi)-\varphi(\pi')||\sim d_{\mathcal{A}}(\pi, \pi')$.

Fix two hyperplanes $\pi, \pi' $, and denote their two corresponding codes by $x(\pi)$ and $x(\pi')$. Then we can present the plane $\pi$ by the equation
\[a_0+b_1t_1+\cdots+b_{d-1}t_{d-1}=t_d, (t_1, \ldots, t_{d-1})\in \mathbb{R}^{d-1},\]
where $x(\pi) = (a_0,b_1,\cdots,b_{d-1})$. Similarly, we can define $\pi'$ from its code coordinate. As presented in the proof of \cite[Lemma 4.3]{HKM} that there is a constant $c$ independent on $\delta$ such that:
\[\pi(\delta) \subset \pi +(\left\lbrace 0 \right\rbrace \times (-c\delta,c\delta)).\]
Therefore, we have
\[ \pi(\delta) \cap \pi'(\delta) \cap \mathcal{C} \subset \left\lbrace (t,u) \in \mathbb{R}^d: u \in B(t_d,c\delta) \cap B(t'_d,c\delta)\right\rbrace, \]
where $\mathcal{C}=\mathcal{C}_{d-1}\times L^1$ and $\mathcal{C}_{d-1}$ is the convex hull of $(0, \ldots, 0), e_1, \ldots, e_{d-1}$, and $L^1$ is the unit segment centered at the origin in $<e_d>$. 

If   $|a_0-a_0'| > \max_{i=1, \ldots, d-1}\left(|b_i-b_i'|\right)+D\delta$, then choosing  $D=2c $ implies $ B(t_d,c\delta) \cap B(t'_d,c\delta)$ is empty. This infers that 
$\pi(\delta) \cap \pi'(\delta) \cap \mathcal{C}=\emptyset$.

If $\max_{i=1, \ldots, d-1}(|b_i-b_i'|)>0$ and $B(t_d,c\delta) \cap B(t'_d,c\delta)\ne \emptyset$, we set $N = \left\lbrace t \in \mathcal{C}_{d-1}: |t_d-t'_d| <2c\delta\right\rbrace$, then we have 
\[N = \left\lbrace t \in  \mathcal{C}_{d-1}: p_-(t) \le t_{d-1} \le p_+(t)\right\rbrace\]
where
\[p_-(t) = p_-(t_1,\cdots,t_{d-2}) = \dfrac{-2c\delta -(a_0-a'_0)-\sum_{i=1}^{d-2}t_i(b_i-b'_i)}{b_{d-1}-b'_{d-1}},\]
and
\[p_+(t) = p_+(t_1,\cdots,t_{d-2}) = \dfrac{2c\delta -(a_0-a'_0)-\sum_{i=1}^{d-2}t_i(b_i-b'_i)}{b_{d-1}-b'_{d-1}}.\]
This implies that $N$ is the intersection of $\mathcal{C}_{d-1}$ and the strip between two parallel hyperplanes $\left\lbrace t_{d-1}=p_-(t)\right\rbrace $ and $\left\lbrace t_{d-1}=p_+(t)\right\rbrace $. By a direct computation, the distance between these two hyperplanes is equal to \[d(p_-(t),p_+(t))=\dfrac{2c\delta}{\sqrt{\sum_{i=1}^{d-2} (b_i-b'_i)^2}} \lesssim \dfrac{\delta}{\max_{i=1, \ldots, d-1}|b_i-b_i'|}.\]

Hence, $N$ is contained in a rectangular box that has the shortest side of length $d\lesssim \delta/w$ and the other $d-2$ sides of length at most $diam(\mathcal{C})=\sqrt{2}$.


To conclude the proof, we do a rotation if needed to assume that the hyperplane $\pi$ has $e_d$ as normal vector. This gives
    \[\pi: f(t) = a_0.\]
    The hyperplane $\pi'$ has the equation of the form
    \[\pi': g(t) = a_0'+b_1't_1+\cdots +b_{d-1}'t_{d-1}.\]
Under these hypotheses, if $(t,u)$ and $(t',u')$ are two elements inside $\pi(\delta)\cap \pi'(\delta)\cap \mathcal{C}$, then triangle inequality gives
\[|u-u'| = |u-f(t)+f(t')-u'| \le |u-f(t)|+|u-f(t')| <2\delta\]
Therefore, the $d$-th coordinate of the intersection part is contained in a box of dimension $\lesssim \delta$. In other words, the intersection $\pi(\delta)\cap \pi'(\delta)\cap \mathcal{C}$ can be covered by roughly $\delta^{2-d}$ dyadic boxes of size $\frac{\delta}{w}\times\underbrace{\delta\times\cdots \delta}_\textrm{ $d-1$ times}$. This completes the proof.
\end{proof}
\begin{lemma}\label{4.2}
    Fix $\pi\in \Pi$. For any $\delta<w\ll 1$, we have 
   \[\sum_{p\in \pi(c\delta)}|I(p)\cap J_{w}(\pi)|\lesssim |J_{w}(\pi)|\cdot |P|\delta^{s} \cdot \frac{1}{w^{\min\{s, 1\}}}\cdot \frac{1}{\delta^{d-2}}.\]
\end{lemma}
\begin{proof}
    Fix $\pi\in J_{w}(\pi)$, then it follows from Lemma \ref{4.1} that $\pi\cap \pi'$ is contained in the union of $\delta^{2-d}$ boxes with parameters $\frac{\delta}{w}\times \delta\times \cdots\times \delta$. This means that the diameter of each box is $\sim \delta/w$. We now bound the above sum in two ways:

Since $P$ is $\delta$-separated, each box contains at most $\lesssim \frac{1}{w}$ elements from $P$. This gives 
\[\sum_{p\in \pi(c\delta)}|I(p)\cap J_{w}(\pi)|\lesssim |J_{w}(\pi)|\cdot \frac{1}{w}\cdot \frac{1}{\delta^{d-2}}.\]

We also observe that each box is contained in a ball of radius $\delta/w$, this infers each box contains at most $C_P |P|\delta^{s}\frac{1}{w^s}$ balls of $P$ since $P$ is $(\delta, s, C_P)$ set. In total, one has 
\[\sum_{p\in \pi(c\delta)}|I(p)\cap J_{w}(\pi)|\lesssim |J_{w}(\pi)|\cdot |P|\delta^{s} \cdot \frac{1}{w^s}\cdot \frac{1}{\delta^{d-2}}.\]
Note that by assumption we have $|P|\delta^{s} \geq 1$ so that combining these two estimates, we get
\[\sum_{p\in \pi(c\delta)}|I(p)\cap J_{w}(\pi)|\lesssim |J_{w}(\pi)|\cdot |P|\delta^{s}\cdot  \frac{1}{w^{\min\{s, 1\}}}\cdot \frac{1}{\delta^{d-2}}.\]
This completes the proof of the lemma. 
\end{proof}
We now continue the proof of the incidence estimate. In particular,
\begin{align*}
    J(\pi) & \lesssim \sum_{p\in \pi(C\delta)}\sum_{i}|I(p)\cap J_{2^i\delta}(\pi)|^{x/y}\\
    &\lesssim  \sum_{i}\sum_{p\in \pi(C\delta)}|I(p)\cap J_{2^i\delta}(\pi)|\cdot |J_{2^i\delta}(\pi)|^{x/y-1}\\
    &\lesssim \sum_{i}|(2^i \delta)^t|\Pi||^{x/y }\cdot |P|\delta^{s} \cdot \frac{1}{(2^i\delta)^{\min\{s, 1\}}}\cdot \frac{1}{\delta^{d-2}}.\\
     \end{align*}
We now fall into the following cases:

If $s\ge 1$ and $tx/y\ge 1$, then 
\[J(\pi)\lesssim |\Pi|^{x/y} \cdot |P|\delta^{s} \cdot \frac{1}{\delta^{d-2}}\cdot \sum_{i}(2^i\delta)^{\frac{tx}{y}-1}\lesssim |\Pi|^{x/y} \cdot |P|\delta^{s} \cdot \frac{1}{\delta^{d-2}}.\]

If $s\ge 1$ and $tx/y < 1$, then
\[J(\pi)\lesssim |\Pi|^{x/y} \cdot |P|\delta^{s} \cdot \frac{1}{\delta^{d-2}}\cdot \sum_{i}(2^i\delta)^{\frac{tx}{y}-1}\lesssim |\Pi|^{x/y} \cdot |P|\delta^{s+\frac{tx}{y}-1} \cdot \frac{1}{\delta^{d-2}}.\]

If $s< 1$ and $tx/y\ge s$, then
\[J(\pi)\lesssim |\Pi|^{x/y} \cdot |P|\delta^{s} \cdot \frac{1}{\delta^{d-2}}\cdot \sum_{i}(2^i\delta)^{\frac{tx}{y}-s}\lesssim |\Pi|^{x/y} \cdot |P|\delta^{s} \cdot \frac{1}{\delta^{d-2}}.\]

If $s< 1$ and $tx/y < s$, then
\[J(\pi)\lesssim |\Pi|^{x/y} \cdot |P|\delta^{s} \cdot \frac{1}{\delta^{d-2}}\cdot \sum_{i}(2^i\delta)^{\frac{tx}{y}-s}\lesssim |\Pi|^{x/y} \cdot |P|\delta^{\frac{tx}{y}-d+2}.\]

Plugging these estimates into $I_{C\delta}(P, \Pi)$, one has 

{\bf Case $1$:} If $s\ge 1$ and $tx/y\ge 1$, then
\[I_{C\delta}(P, \Pi)\lesssim |P||\Pi|\delta^{\frac{y}{x+y}\left(s-d+2\right)}.\]
If $s-d+2<0$ then we only have the trivial upper bound. Thus, this upper bound is valid in the range $s-d+2>0$. By choosing $x=y$, we obtain the upper bound
\[I_{C\delta}(P, \Pi) \lesssim |P||\Pi|\delta^{\frac{1}{2}\left(s-d+2\right)}. \]

{\bf Case $2$:} If $s\ge 1$ and $tx/y<1$, then
\[I_{C\delta}(P, \Pi)\lesssim |P||\Pi|\delta^{\frac{y}{x+y}\left(s+\frac{tx}{y}-d+1\right)}.\]
We observe that $t\le 1$ since $x\ge y$. If $s-d+2<0$, then we only have the trivial upper bound. If $s-d+2>0$, then we can choose 
$y=xt+\epsilon$, $\epsilon > 0$, to get the upper bound \[I_{C\delta}(P, \Pi)\lesssim |P||\Pi|\delta^{\frac{t}{1+t}\left(s-d+2\right) - \epsilon}.\]
{\bf Case $3$:} If $s< 1$ and $tx/y\ge s$, then
\[I_{C\delta}(P, \Pi)\lesssim |P||\Pi|\delta^{\frac{y}{x+y}\left(s-d+2\right)}.\]

In this case, since $s < 1$ and $d \geq 3$, we have $s-d+2$ is always negative, so only trivial upper bound is obtained.

{\bf Case $4$:} If $s< 1$ and $tx/y< s$, then
\[I_{C\delta}(P, \Pi)\lesssim |P||\Pi|\delta^{\frac{y}{x+y}\left(\frac{tx}{y}-d+2\right)}.\]

Since $tx/y<s$, $s<1$, and $d\ge 3$, we only have the trivial upper bound in this case.

Putting these cases together, Theorem \ref{CSbound} follows.
\section{Sharpness of Theorem \ref{main-theorem}}

In this section we provide an example to show the sharpness of our result in Theorem \ref{main-theorem}. The construction is mainly adapted from the two dimensional example due to Fu and Ren in \cite{FuRen}. For the reader's convenience, we sketch the ideas here. 

We first recall that Construction 4 in \cite{FuRen} shows that in the plane when $s + t \geq 3$, the incidence between balls and tubes is $\sim \delta^{-(s +t -1)} =\delta |P||\mathcal{T}|$. We will see that their example can be extended to higher dimensions which gives us a sharp upper bound. 

Note that the incidences results of Fu and Ren \cite{FuRen} in the plane are applied to the sets of balls and tubes satisfying Definition \ref{def:kt} (Katz-Tao $(\delta, s)$ sets). However, as we mentioned before, when the set is of size $ \sim \delta^{-s}$ then Definitions \ref{df:our} and \ref{def:kt} are equivalent. Moreover, in their sharpness example, what they constructed are actually $(\delta, s, C)$ sets. Therefore, we can extend their construction to higher dimensional spaces.  In the rest, we  denote $D=\delta^{-1}$ for convenience.

\subsection{Construction in $\mathbb{R}^3$}

The idea is to construct a configuration such that the intersection between this configuration and the plane $O_{xy}$ is exactly Construction 4 in \cite{FuRen}. Thus, we can reduce to the two-dimensional case when we consider the intersection between $\delta$-hyperplanes and $\delta$-balls with $O_{xy}$.  Then we move the plane $O_{xy}$ vertically by spacing $2\delta$ (two consecutive planes have distance $2\delta$) to get the desired bound. 

We consider balls of form: 
\[ (x-a_1)^2 +(y-a_2)^2+z^2=\delta^2.\]
Clearly these are $\delta$-balls in $\mathbb{R}^3$ whose intersection with $O_{xy}$ are also $\delta$-balls in dimension two. For each tube in Construction 4 in \cite{FuRen}, we just extend the tube vertically to the $\mathbb{R}^3$ which becomes a plane. Thus, if we put balls by spacing $2\delta$ in $\mathbb{R}^3$ (the centers of two consecutive balls have distance of $2\delta$), then the incidence between these balls and planes is $ \sim \delta^{-(s+t -1)} \cdot \delta^{-1} = \delta \cdot \delta^{-s-1} \cdot \delta^{-t}$ which is equal to $\delta |P||\Pi|$.

The next step is to check that the set of balls and hyperplanes constructed above satisfy Definition \ref{df:our}. Since the hyperplanes are just the tubes expanding vertically, the number of $\delta$-hyperplanes is equal to that of $\delta$-tubes, i.e, $\sim D^t$. Moreover, as already shown in the paper of Fu and Ren \cite{FuRen} (page 6) that the set of $\delta$-tubes is a $(\delta, t, C)$ set which gives that our set of $\delta$-hyperplanes is also a $(\delta, t, C)$ set.

Now for each copy of $O_{xy}$, there are $\sim D^{s}$ balls, as stated in \cite{FuRen}, and there are $\delta^{-1}$ copies in total. So $|P| \sim D^{s+1}$. Moreover, as shown in \cite{FuRen}, for each copy of $O_{xy}$, a ball of radius $w$ in the plane contains at most $Cw^{s}\delta^{-s}$ $\delta$-balls. So, a ball $B_w$ in $\mathbb R^3$ contains at most $Cw^{s}\delta^{-s-1}$ $\delta$-balls in $P$. This gives that
 \[|\{p \in P: p \subset B_w\}| \le C w^{s}\delta^{-s-1}\sim Cw^{s}|P|.\]
Therefore, the set $P$ is a $(
\delta,s,C)$-set.
\subsection{Construction in $\mathbb{R}^d$, $d\ge 4$}
For dimension $d>3$, the construction is inductively constructed. For instance, we consider $d=4$. We directly extend each hyperplane from above 3-dimensional construction to a hyperplane in $\mathbb R^4$ by just adding a variable $x_4$. Precisely, assume a hyperplane in above 3-dimensional construction is defined by $\{(x_1, x_2, x_3): ax_1+bx_2+cx_3=e\}$, then the extension in $\mathbb{R}^4$ is defined by $\{(x_1, x_2, x_3, x_4): ax_1+bx_2+cx_3 +0x_4=e\}.$ We now consider balls defined by the equation:
\[ (x_1-a_1)^2 +(x_2-a_2)^2+x_3^2+x_4^2=\delta^2.\] 
These balls are moved by spacing $2\delta$ in $x_3$ and $x_4$ directions. Thus, we see that when we fix the last coordinate $x_4$, the number of incidences between these hyperplanes and balls is the same as we constructed above in 3-dimension which is $\delta \delta^{-s-1} \delta^{-t}$. Moreover, we have $\delta^{-1}$ number of copies in $x_4$ direction which implies that the number of incidences is $ \delta \delta^{-s-1} \delta^{-t} \delta^{-1} =\delta |P||\Pi|$. It is not difficult to prove that the sets of balls and hyperplanes are $(\delta, s, C)$ and $(\delta, t, C)$ sets, respectively. For other dimensions $d > 4$, the construction works in the same way. Thus, the number of incidences 
$I(P,\Pi) \sim |\{\text{incidences in the plane}\}|\cdot \delta^{2-d} \sim \delta |P||\Pi|$. 

For the sharpness of Theorem \ref{CSbound}, one might think of the same idea by extending Constructions 1, 2, 3 in \cite{FuRen} to higher dimensions, but it does not match the upper bounds in Theorem \ref{CSbound} at least in the way we tried.

\section{Acknowledgements}
T. Pham would like to thank the Vietnam Institute for Advanced Study in Mathematics (VIASM) for the hospitality and for the excellent working condition. C.-Y. Shen was partially supported by NSTC grant 111-2115-M-002-010-MY5.


\begin{thebibliography}{}
\bibitem{PB}
P. Bradshaw,\textit{ An incidence result for well-spaced atoms in all dimensions}, to appear in 
Journal of the Australian Mathematical Society, 2023. 

\bibitem{DOV}
D. Dabrowski, T. Orponen, and M. Villa, \textit{Integrability of orthogonal projections, and applications to Furstenberg sets}, Advances in Mathematics, \textbf{407}(2022): 108567.


\bibitem{Alex} S. Eswarathasan, A. Iosevich, and K. Taylor, \textit{Fourier integral operators, fractal sets,
and the regular value theorem}, Advances in Mathematics, \textbf{228}(2011) 2385–2402.


\bibitem{FuRen} Y. Fu and K. Ren,  \textit{Incidence estimates for $\alpha $ dimensional tubes and $\beta$ dimensional balls in $\mathbb R^2$}, arXiv:2111.05093, 2021.



\bibitem{GSW}
L. Guth, N. Solomon, and H. Wang, \textit{Incidence estimates for well spaced
tubes}, Geometric and Functional Analysis, \textbf{29}(6):1844–1863, 2019.

\bibitem{HKM}
K. Hera, T. Keleti, and A. Mathe, \textit{Hausdorff dimension of unions of affine subspaces and of Furstenberg-type sets}, Journal of Fractal Geometry, \textbf{6}(3)(2019), pp. 263--284.




\bibitem{HPab}
K. H\'{e}ra, P. Shmerkin, and A. Yavicoli, \textit{An improved bound for the dimension of $(\alpha, 2\alpha) $-Furstenberg sets}, Revista Matemática Iberoamericana, \textbf{38}(1)(2021): 295-322.


\bibitem{III}
A. Iosevich, H. Jorati, and I. Laba, \textit{Geometric incidence theorems via Fourier analysis}, Transactions of the American Mathematical Society, \textbf{361}(12)(2009): 6595-6611.


\bibitem{KatzTao} N. Katz and T. Tao, \textit{Some connections between Falconer's distance set conjecture and sets of Furstenburg type}, New York J. Math., \textbf{7}(2001) 149--187.





\bibitem{M95}
P. Mattila, \textit{Geometry of sets and measures in Euclidean spaces}, Fractals and Rectifiabilit, Cambridge University Press, 1999.


\bibitem{motmot}
T. Orponen, P. Shmerkin, and H. Wang, \textit{Kaufman and Falconer estimates for radial projections and a continuum version of Beck's Theorem}, arXiv:2209.00348, 2022.



\bibitem{PhS} D. H. Phong and E. Stein, \textit{Hilbert integrals, singular integrals, and Radon transforms}, Acta Math \textbf{(157)}(1986), 99--157.









\end{thebibliography}
\end{document}